\documentclass[12pt]{amsart}

\usepackage{latexsym, amssymb, amsmath}

\usepackage{amsfonts, graphicx}
\usepackage{graphicx,color}
\newcommand{\be}{\begin{equation}}
\newcommand{\ee}{\end{equation}}
\newcommand{\beq}{\begin{eqnarray}}
\newcommand{\eeq}{\end{eqnarray}}

\newtheorem{thm}{Theorem}[section]

\newtheorem{lma}{Lemma}[section]
\newtheorem{prop}{Proposition}[section]

\newtheorem{claim}{Claim}[section]

\newtheorem{defn}{Definition}[section]
\theoremstyle{remark}
\newtheorem{rem}{Remark}[section]
\numberwithin{equation}{section}
\def\tr{\operatorname{tr}}

\def\be{\begin{equation}}
\def\ee{\end{equation}}
\def\bee{\begin{equation*}}
\def\eee{\end{equation*}}

\def\lf{\left}
\def\ri{\right}

\def\K{K\"ahler }

\def\KE{K\"ahler-Einstein }
\def\KR{K\"ahler-Ricci }
\def\Ric{\text{\rm Ric}}
\def\Rm{\text{\rm Rm}}

\def\wt{\widetilde}

\def\p{\partial}

\def\heat{\lf(\frac{\p}{\p t}-\Delta\ri)}
\def\tr{\operatorname{tr}}

\def\e{\epsilon}
\def\a{{\alpha}}
\def\b{{\beta}}

\def\ijb{{i\bar{j}}}

\def\R{\mathbb{R}}
\def\C{\mathbb{C}}

\def\ddb{\sqrt{-1}\partial\bar\partial}

\begin{document}

 \title{\bf The K\"AHLER Ricci flow around complete bounded curvature K\"AHLER metrics}\vskip .2cm
\author{Albert Chau$^1$}
\address{Department of Mathematics,
The University of British Columbia, Room 121, 1984 Mathematics
Road, Vancouver, B.C., Canada V6T 1Z2} \email{chau@math.ubc.ca}

\author{Man-Chun Lee}

\address{Department of Mathematics,
The University of British Columbia, Room 121, 1984 Mathematics
Road, Vancouver, B.C., Canada V6T 1Z2} \email{mclee@math.ubc.ca}

\thanks{$^1$Research
partially supported by NSERC grant no. \#327637-06}

%\thanks{$^2$Research partially supported by Hong Kong RGC General Research Fund  \#CUHK 403108}

\begin{abstract} We produce complete bounded curvature solutions to K\"ahler-Ricci flow with existence time estimates,  assuming only that the initial data is a smooth \K metric uniformly equivalent to another complete bounded curvature \K metric.  We obtain related flow results for non-smooth as well as degenerate initial conditions.  We also obtain a stability result for complex space forms under the flow.

\end{abstract}

\maketitle\markboth{Albert Chau and Man-Chun Lee} {The \K Ricci flow around a complete bounded curvature \K metric}

\section{Introduction}

 Let $M^n$ be a non-compact complex manifold.  The \K Ricci flow on $M^n$ starting from an initial \K metric $g_0$ is the evolution equation 
 \be\label{krf}
 \left\{
   \begin{array}{ll}
     \displaystyle\frac{\partial g_{\ijb}}{\partial t} =-R_{\ijb}\\
          g(0)  = g_0.
   \end{array}
 \right.
 \ee
By a solution to \eqref{krf} we mean a smooth family of \K metrics $g(t)$ satisfying  \eqref{krf} on $M\times[0, T)$ for some $T>0$.   A classical theorem of W.X. Shi \cite{Shi1997} says that if $g_0$ is complete with bounded curvature then \eqref{krf} has a solution $g(t)$ which has bounded curvature and is equivalent to $g_0$  for $t>0$.   In this paper we show that the same result holds assuming only that $g_0$ is equivalent to a complete \K metric with bounded curvature.  In particular, $g_0$ may have unbounded curvature.

 There have been many other works on the existence of solutions to \K Ricci flow and real (Riemannian) Ricci flow when the initial metric $g_0$ has possibly unbounded curvature (see for example \cite{CabezasWilking2015} \cite{CabezasBamlerWilking2017}, \cite{GiesenTopping2013}, \cite{CLT1},  \cite{CLT2},  \cite{He2016}, \cite{LeeTam2017}, \cite{Simon2002}, \cite{Hochard2016} and references therein).  In particular, Simon \cite{Simon2002} proved the real Ricci flow has a complete solution starting from any Riemannian metric $g_0$ inside an $\e$ neighborhood of a complete bounded curvature Riemannian metric $h$ in the $C^0$ topology induced by $h$, provided $\e$ is sufficiently small depending on $n$.   A solution to \K Ricci flow is a special solution to the real Ricci flow, thus our result shows the smallness condition on $\e$ can be removed in Simon's theorem when $g_0$ and $h$ are smooth \K.

 Fix a complete \K metric $h$ on $M$ with curvature bound  $|Rm(h)|\leq K$.  For $c_2>c_1>0$, define the following space of \K metrics 
\begin{equation}\label{S}
S(c_1, c_2,h):=\{ g_0: g_0 \text{ is smooth and \K on $M$ and } c_1h \leq g_0 \leq c_2 h  \} 
\end{equation}  
and for each $g_0\in S(c_1, c_2, h)$, define
\begin{equation}\label{defnexistencetime}
\begin{split}
T_{g_0}=\sup \{T:& \text{\eqref{krf} has a solution $g(t)$ on $M\times[0, T)$ which has}\\
&\text{bounded curvature and is equivalent to $g_0$ for all $t>0$}\}
\end{split}
\end{equation}
We will refer to a solution as in the definition of $T_{g_0}$ as a bounded curvature solution equivalent to $g_0$.  Then our main existence result for \eqref{krf} can be stated as follows.

\begin{thm}\label{mainthm}
If $g_0\in S(1, c, h)$ then \eqref{krf} has a bounded curvature solution $g(t)$ on $M\times [0, T_h)$ which is equivalent to $g_0$.   Moreover, there exists positive constants $a(n, c, K), T(n, c, K),C_2(n, c, K),  C_1(n)>0$ such that for all $t\in [0, T)$ we have
\begin{enumerate}
\item[(1)] $g(t)\in S((e^{-C_1 Kt}), C_2, h)$ 
\item[(2)]$
\sup_M \| Rm(t)\|_{g(t)} \leq a/t
$
\end{enumerate}
\end{thm}
\begin{rem}
The existence part also holds when the reference metric $h$ is only Hermtian (not K\"ahler) and with bounded torsion and Chern curvature in addition. We leave the verification of this to interested readers.
\end{rem}
%\begin{rem}
 %The fact that we may have $C c_1 h \leq g(t)$ say as in (1) of Theorem \ref{mainthm}, follows from the fact that on p4, the constant $a_2$ there is actually coming only from the lower bound $a_2^{-1}h\leq g_0$ (and not from the upper bound $g_0\leq a_2 h$).  I still have modify the proof of that Lemma accordingly..
%\end{rem}

 When $g_0$ is Hermitan but not \K on $M$ then by a smooth solution to \eqref{krf} on $M\times (0, T)$ we mean a smooth family of \K metrics $g(t)$ on $M\times (0, T)$ solving the first equation in \eqref{krf} such that  $g(t) \to g_0$ pointwise on $M$ as $t\to 0$.   In this sense, we may use Theorem \ref{mainthm} to solve \eqref{krf} assuming $g_0$ is a $C^2_{loc}(M)$ limit of smooth \K metrics $g_k\in S(c, c_k, h)$ for some $c, c_k>0$ where $c_k$ is not assumed uniformly bounded above.  On the other hand, if $g_0$ is a $C^0_{loc}(M)$ limit of smooth \K metrics $g_k\in S(c, \tilde{c}, h)$ for some $c, \tilde{c}>0$ we may have a solution to \eqref{krf} satisfying the same conclusions as in Theorem \ref{mainthm}.

\begin{thm}\label{non-smooth}  
We have the following

\begin{enumerate}
\item 
If $g_0\in Cl^2_{loc}  \left( \bigcup_{c>0} S(1, c, h)\right)$, then there is $T(n,K), C_1(n) >0$ such that \eqref{krf} has a solution on $M\times (0,T)$ in the sense above with 
$g(t)\geq e^{-C_1Kt}h$ on $(0,T]$. 
\item
 If $g_0\in Cl^0_{loc} S(1,c,h)$ for some $c>1$, then \eqref{krf} has a bounded curvature solution $g(t)$ on $M\times(0, T_h)$ satisfying the same conclusions as in  Theorem \ref{mainthm}.
\end{enumerate}
\end{thm}

%begin{rem}
%If $g_0\in Cl^0_{loc} S(1,c,h)$ for some $c>1$, then the conclusion in Theorem \ref{mainthm} also hold where $g(t)$ %attains the initial metric in the sense of above.
%\end{rem}

As another application, we have the following stability result for \eqref{krf} around complex spaceforms. In the Riemannian case, the stability of spaceforms was first studied in \cite{SSS1,SSS2}.  We show that in the \K category, we do not require the $\e$-fairness required in those works.
\begin{thm}\label{Stability} Suppose $(M,h)$ is a complete noncompact \K manifold with constant holomorphic sectional curvature $H_h=2k\leq 0$. Then if $g_0\in S(c_1,c_2,h)$ for some $c_2>c_1>0$, then  $T_{g_0}=+\infty$. Moreover, 
\begin{enumerate}
\item If $k<0$, then $t^{-1}g(t)$ converges to $h$ in $C^\infty_{loc}$ as $t\rightarrow \infty$;
\item If $k=0$, then $g(t)$ converges sub-sequentially in $C^\infty_{loc}$ to a complete flat metric as $t\rightarrow \infty$.
\end{enumerate}
\end{thm}

 On the other hand, if $g_0$ is assumed only degenerate K\"ahler, in other words the corresponding (1,1) form $\omega_0$ is  closed and nonnegative but not necessarily positive, then under certain condtitions we may also produce a solution to \eqref{krf} with exsitence time estimates in the sense of the following Theorem

\begin{thm}\label{instantaneous}
Suppose $(M,\omega_h)$ is a complete noncompact \K manifold with bounded curvature. If $\omega_0$ is a closed nonnegative $(1,1)$ form such that 
\begin{enumerate}
\item $\omega_h\geq \omega_0$ on $M$;
\item $\omega_0-s\Ric(\omega_h) +s\ddb f> \b\omega_h$ for some $\b,s>0$ and $f\in C^\infty(M)\cap L^\infty(M)$.
\end{enumerate}
Then there is a smooth solution to 
\begin{align}\label{potential-flow-equ}
\dot\varphi=\log \frac{(\omega_0-t\Ric(\omega_h)+\ddb \varphi)^n}{\omega_h^n}
\end{align}
on $M\times (0,s)$ with $\varphi\rightarrow 0$ in $L^\infty(M)$ and $\omega(t)=\omega_0-t\Ric(\omega_h)+\ddb \varphi$ is a solution to the \KR flow \eqref{krf} on $M\times (0,s)$ which is  uniformly equivalent to $\omega_h$ on $M$. Moreover $\varphi$ is smooth up to $t=0$ on $U=\{ x: \omega_0>0\}$.
\end{thm}
In particular, if $h$ is a \KE metric with bounded curvature and negative Einstein constant, then any \K metric bounded from above by some multiples of $h$ can be deformed to $h$ along the normalized \KR flow. When $n=1$, the existence of Ricci flow starting from incomplete metric has been studied in details by Giesen and Topping \cite{GiesenTopping2010,GiesenTopping2011,GiesenTopping2013} where they do not require any boundedness on $\omega_0$. In contrast with \cite[Theorem 1.1]{HuangLeeTam2019}, we remove the boundedness of $|\nabla^h g_0|$ when $h$ has bounded curvature and $g_0$ is K\"ahler.

 Our approach in proving Theorem \ref{mainthm} is similar in spirit to that taken in \cite{LeeTam2017} (see also \cite{SimonTopping2017,Hochard2016}). We make use of an iterative process using the Chern-Ricci flow to produce a solution $g_R(t)$ to \eqref{krf} on $B_{h}(p, R)\times[0, T)$ with $g_R(0)=g_0$.  Estimates for $g_R(t)$ over a compact subset $S\subset \subset B_{h}(p, R)$ are then established which depend only on $n,  c_1, c_2, K, S$ thus allowing us to let $R\to \infty$ to obtain a smooth limit solution $g(t)$ on $M\times [0, T)$ with the desired properties.

 The paper is orgainzed as follows.  In \S 2 we prove our main local estimates for \K Ricci flow.  In \S 3 we recall some basic results for Chern Ricci flow from \cite{LeeTam2017}.  Then in \S 4, 5, 6, 7 we prove Theorems \ref{mainthm}, \ref{non-smooth}, \ref{Stability} and \ref{instantaneous} respectively.

 %\section{pre-stuff}

%\begin{thm}\label{KRF-equ}
%Suppose $(M,h)$ is a complete noncompact \K manifold with bounded curvature. If $g_0$ is another complete \K metric uniformly equivalent to $h$, then it admits a short time solution to the \KR flow on $M$ which is uniformly equivalent to $h$.
%\begin{rem}
%It is likely that $h$ can be a non-\K metric with extra assumptions on its torsion.
%\end{rem}
%\end{thm}

{\it Acknowledgement}: The authors would like to thank Luen-Fai Tam and Xiangwen Zhang for the interest in this work.

\section{estimates for \K Ricci flow}

 In this section we establish basic estimates for solutions to \eqref{krf}.   Here, $(M, g_0)$ is a smooth \K manifold and $h$ is another smooth \K metric on $M$ with bisectional curvatures bounded by $K=1$.  Theorem \ref{local-Lip}, Lemma \ref{curv-esti} and Lemma \ref{l-curv1} are purely local in nature and $B_{g_0}(p, r)$ there refers to an arbitrary ball of radius $r$ relative to $g_0$. Then in Proposition \ref{prop1} we establish a global estimate for smooth solutions  to \eqref{krf} assuming a priori they are uniformly equivalent to $h$ for all $t$.

 Given a solution $g(t)$ to \eqref{krf} we will denote by $\omega(t)$ the corresponding family of \K forms.   We will also let $\omega_h$ denote the \K form corresponding to $h$.  We will let 

 $$\varphi(t)=\int^t_0 \log\frac{\omega^n (s)}{\omega_h^n}\;ds.$$
In particular, applying $\ddb$ to both sides and using \eqref{krf} and the local formula for the Ricci tensor of a \K metric allows us to write
\begin{align}
\omega(t)=\omega_0 -t\Ric(h)+\ddb \varphi.
\end{align}

 We also recall the following evolution equations for \eqref{krf} (see for example \cite{SongWeinkove2013}).   Here $\Psi_{ij}^k=\Gamma(g(t))_{ij}^k-\Gamma(h)_{ij}^k$.  
\begin{equation}\label{genereal-equ}
\begin{split}
\heat tr_gh &=-g^{i\bar j} g^{p\bar q}h_{k\bar l} \Psi_{pi}^k \Psi_{\bar q\bar j}^{\bar l}+g^{i\bar j}g^{p\bar q}  R^h_{i\bar jk\bar l};\\
\heat |\Psi|^2&=-|\nabla \Psi|^2-|\bar\nabla \Psi|^2-2{\bf Re}\left(g^{r\bar s}g^{i\bar j}g^{p\bar q}g_{k\bar l}\Psi_{\bar j\bar q}^{\bar l}\nabla_r \hat R_{i\bar sp}^k  \right);\\
\heat |\Rm|^2&\leq -|\nabla \Rm|^2-|\bar\nabla \Rm|^2+C_n|\Rm|^3.
\end{split}
\end{equation}

 For notational convenience, we adopt the notation $a\wedge b=\min\{a,b\}$ for any real numbers $a, b$ below.  
The main estimate in this section is the following

\begin{thm}\label{local-Lip}
For any $a,\lambda>1$, there is $\tilde T(n,a,\lambda)>0$ and $C_0(n),c_1(n)>1$ such that the following is true. Let $g(t)$ be a solution to the \KR flow on $B_{g_0}(p,1)\times [0,T]$ so that
\begin{enumerate}
\item $|\Rm(h)|\leq 1$ on $B_{g_0}(p,1)$;
\item $\lambda^{-1}h\leq g_0\leq \lambda h$ on $B_{g_0}(p,1)$.
\item $|\mathrm{Rm}(g(t))|\leq a t^{-1}$ on $B_{g_0}(p,1)\times (0,T]$;
\item $|\varphi|\leq at$ on $B_{g_0}(p,1)\times [0,T]$;
\end{enumerate} 
Then for all $t \in [0,T \wedge \tilde T]$ we have $B_t(p,1/4)\subset B_0(p,1)$, and for all $x\in B_t(p,1/4)$ we have $$C_0^{-1} \lambda ^{-1} h\leq g(t)\leq   C_0 \lambda^{c_1}h.$$
\end{thm}
\begin{proof}
We begin by constructing appropriate barrier functions for our arguments.  By the shrinking ball lemma \cite[Lemma 3.2]{SimonTopping2016}, for $t\in [0,T]$
\begin{align}\label{shrinkingball}
B_{g(t)}(p,1-\b_n \sqrt{at})\subset B_{g_0}(p,1)
\end{align}
for some constant $\b_n$.  By \cite[Lemma 8.3]{Perelman2002} with $\displaystyle K=\frac{a}{t}$ and $\displaystyle r=\sqrt{\frac{t}{a}}$ there, we may infer that 
\begin{align}\label{E1}
\heat \left[d_t(x,p)+c_1(n)\sqrt{at}\right] \geq 0
\end{align}
 in the sense of barriers whenever $d_t(x,p)\geq \sqrt{a^{-1}t}$.  Here $d_t(x,p)=dist_{g(t)}(x, p)$.  Denote $\eta(x,t)=d_t(x,p)+\tilde \b_n\sqrt{a t}$ where $\tilde \b_n\geq \max\{c_1(n),\b_n\}$. Let $\phi$ be a smooth function on $[0,+\infty)$ so that $\phi\equiv 1$ on $[0,\frac{3}{4}]$, vanishes outside $[0,1]$ and satisfies $-100{\phi}^{3/4}\leq \phi'\leq 0;  \phi''\geq -100\phi $. 

 We will make use of the evolving barrier $\log \Phi(x, t):=\log \phi\left(\eta(x,t) \right)$.   By \eqref{shrinkingball} and \eqref{E1}, we may choose $\tilde T$ sufficiently small depending on $n, a$ such that for all $t\in T\wedge \tilde T$,

\begin{enumerate}
\item $Domain( \log \Phi(x, t))\subset B_{g_0}(p,1)$ 
\item
$B_{g(t)}(p,1/2) \subset  \{x: \log \Phi(x, t)=0\}$
\item
$\heat (\log \Phi ) \leq C_n /\Phi^2$ 

\end{enumerate}
We will also make use of the cut off function $\wt \Phi(x,t)=\phi \left(2 \eta(x,t) \right)$.  Then similarly we may choose $\tilde T$ sufficiently small depending on $n, a$ such that for all $t\in T\wedge \tilde T$,

\begin{enumerate}
\item $B_{g(t)}(p,1/2)\subset Domain(\wt \Phi (x, t))\subset  \{x:\log \Phi(x, t)\neq -\infty\}$
\item $B_{g(t)}(p,1/4) \subset  \{x: \wt \Phi(x, t)=1\}$
\item
$\heat (\wt \Phi )  \leq  C_n \wt \Phi$
\item $\frac{| \partial \wt \Phi |^2}{\wt \Phi^2} \leq C_n\wt \Phi^{-1/2}$ 
\end{enumerate}

The function $d_t(x,p)$ and hence barrier functions above are in general only Lipschitz continuous on $M$, though by Calabi's trick (see \cite[Section 7]{SimonTopping2016} for a detailed exposition) we may assume these to be smooth, and the differential inequalitites above to hold in the usual sense, at a given point where we argue by the maximum principle below.

\begin{claim}\label{sub-lma-1}
If $\tilde T(n,a,\lambda)$ is sufficiently small, then for all $x\in B_t(p,1/2), \;t\leq T\wedge\tilde T$,
\begin{align}\label{local-lip-bound}
\frac{1}{C(n)\lambda}h\leq  g(t).
\end{align}
\end{claim}
\begin{proof}[proof of Claim \ref{sub-lma-1}]

Consider the function 
\begin{align}
F(x,t)=\log \tr_gh+2\log \Phi-L\varphi+nLt (\log t -2)
\end{align}
for times $t\in [0,T \wedge \tilde T\wedge (nL)^{-1}]$ where $\Phi, \tilde T$ is as above and $L=6\lambda$.  By the properties of  the barrier $\log \Phi$, $F(x, t)$ attains a maximum value at some point $(x_0,t_0)$ in its domain.  If $t_0 =0$, then the Claim immeditely follows from the definition of $F$ and the hypothesis of the Theorem.  Now suppose $t_0 >0$.  Assuming that $ \tilde T \leq 1/2\lambda$, we have 
\begin{align}\label{ref-ineq}
\omega(t)=\omega_0-t\Ric(h)+\ddb \varphi \geq \frac{1}{2\lambda} \omega_h+\ddb\varphi, 
\end{align}
and then the evolution equation of $\log\tr_hg$ from \eqref{genereal-equ} gives the following at $(x_0,t_0)$
\begin{equation}
\begin{split}
 \heat F
&\leq  \tr_gh+\frac{C_n}{\Phi^2}+L \log \frac{\det h}{\det g}+L\Delta \varphi+nL\log t-nL\\
&\leq \tr_gh\cdot  \left(1-\frac{1}{2\lambda }L \right)+nL\log (t\tr_gh)+\frac{C_n}{\Phi^2}\\
&\leq -\tr_gh +\frac{C_n}{\Phi^2}
\end{split}
\end{equation}
  Here we have assumed that $tr_g h (x_0,t_0)\geq 1$ (without loss of generality) in the second inequality, and we have used the elementary inequality $\log x\leq x$ for all $x>0$ and $L=6\lambda$ in the last inequality.  By the maximum principle we then conclude that $ \tr_gh \leq \frac{C_n}{\Phi^2}$ at $(x_0,t_0)$ and thus
\begin{align}
F(x_0,t_0)\leq C_n+aLt_0+nLt_0(\log t_0-2)\leq \tilde C(n)
\end{align}
for some constant $ \tilde C(n)$ provided we further shrink $\tilde T$ if necessary depending only on $n, a, \lambda$.

On the other hand, we also have $F(0)\leq \log \lambda+\log n$.  Thus in summary, we conclude that for $t\in [0,T \wedge \tilde T\wedge (nL)^{-1}]$ we have
\begin{align}
F(x,t)\leq \max\{ \tilde C_n,\log \lambda+\log n\}
\end{align}
and the claim follows from the defninition of $F$ and the properties of the barrier $\log \Phi$.
\end{proof}

 To prove the Theorem, it suffices to obtain a local upper bound for the volume form of $g(t)$ in view of Claim \ref{sub-lma-1}.  We do this in the following

\begin{claim}\label{sub-lma-2}
There exists  $c(n), \tilde T(n,a,\lambda)>0$ such that for all $x\in B_t(p,1/4), \;t\leq T\wedge\tilde T$ we have
\begin{align}
\dot\varphi =\log \frac{\omega^n(t)}{\omega_h^n}\leq  c(n)\log \lambda.
\end{align}
\end{claim}

\begin{proof}[proof of Claim \ref{sub-lma-2}]

Consider the function 
$$G= \wt\Phi\frac{(\dot\varphi)^2}{1+L\varphi}$$
for $t\leq T\wedge \tilde T$ where $L>0$ is a constant to be chosen later.  Due to the cutoff function, $G$ attains its maximum at some point $(x_0,t_0)$ with $t_0\leq T\wedge \tilde T$. If $t_0=0$, then the upper bound in the claim is trivial. Now suppose $t_0>0$.   Assume $\tilde T \leq 1/(2\lambda)$ so that $\eqref{ref-ineq}$ gives $\hat g(t_0):=g_0-t_0\Ric (h)\leq h/(2\lambda)$. We calculate at $(x_0, t_0)$ that

\begin{equation}\label{evo-G}
\begin{split}
&\quad \heat \frac{(\dot\varphi)^2}{1+L\varphi}\\
&=-\frac{L(\dot\varphi)^2 }{(1+L\varphi)^2}\left(\dot\varphi-n+\tr_g \hat g \right)-\frac{2L^2\dot\varphi^2}{(1+L\varphi)^3}g^{i\bar j} \varphi_i \varphi_{\bar j}\\
&\quad -\frac{2\dot\varphi }{1+L\varphi}\tr_g (\Ric (h))-\frac{2 g^{i\bar j} \dot\varphi_i\dot\varphi_{\bar j} }{1+L\varphi}+\frac{4L\dot\varphi }{(1+L\varphi)^2}{\bf Re}\left(g^{i\bar j} \dot\varphi_i \varphi_{\bar j}\right)\\
&\leq -\frac{L(\dot\varphi)^2 }{(1+L\varphi)^2}\left(\dot\varphi-n+\tr_g \hat g \right)+\frac{2|\dot\varphi|}{1+L\varphi}\tr_gh\\
&\leq -\frac{L\dot\varphi^3}{(1+L\varphi)^2}+\frac{Ln\dot\varphi^2}{(1+L\varphi)^2}+\left[-\frac{L\dot\varphi^2}{2\lambda(1+L\varphi)^2}+\frac{2|\dot\varphi|}{1+L\varphi}\right]\tr_gh.
\end{split}
\end{equation}
By Claim \ref{sub-lma-1} and the properties of the cut off function $\wt \Phi$ we have the lower bound
\begin{align}
\dot\varphi=\log \frac{\det g}{\det h} \geq -n\log (2n\lambda).
\end{align}
on the support of $G$.  Therefore, we will assume that $\dot\varphi(x_0,t_0) \geq 1$ as the otherwise the Claim follows.
We may then choose $L=16\lambda$ and assume $\tilde T$ is sufficiently small depending only on $n, a, \lambda$ so that the last term in \eqref{evo-G} will be bounded above by $0$, and moreover, using $\nabla G (x_0, t_0)=0$ and the properties of $\wt \Phi$, that at $(x_0, t_0)$
\begin{equation}\label{E0}
\begin{split}
0&\leq \heat G \\
&\leq-\frac{L\dot\varphi^3}{(1+L\varphi)^2}\wt\Phi+\frac{Ln\dot\varphi^2}{(1+L\varphi)^2}\wt\Phi+C_nG+2\frac{|\partial  \wt\Phi|^2}{ \wt\Phi^2}G\\
&\leq -G^{3/2}\wt\Phi^{-1/2}+C_nG+ C_n \wt\Phi^{-1/2}G
\end{split}
\end{equation}
which in turn implies $G(x_0,t_0)\leq C(n)$ for some constant $C(n)$.  Thus $G(x, t) \leq C(n)(\log \lambda)^2$ on $M\times[0, T\wedge \tilde T]$, and by assuming $\tilde T$ is sufficiently smaller still depending on $n, a, \lambda$, by the properties of the cut off $\wt \Phi$ we may conclude that for  $x\in B_t(p,1/4)$ and $t\leq T\wedge \tilde T$ we have
\begin{align}
\dot\varphi \leq C_n\log \lambda
\end{align}
\end{proof}

The Theorem follows from combining Claim \ref{sub-lma-1} and Claim \ref{sub-lma-2} together with the following elementary inequality
$$\tr_hg \leq \frac{\det g}{\det h}(\tr_hg)^{n-1}.$$
\end{proof}

The next lemma shows that we have a local curvature estimate provided $g(t)$ stays uniformly equivalent to a good reference metric. In \cite{ShermanWeinkove2012}, Sherman-Weinkove showed essentially the same estimate but with less detail on the dependence of the various constants in the Lemma.
\begin{lma}\label{curv-esti}
For any $a,\Lambda>1$, there is $C_1(n,\Lambda),\hat T(n,a,\Lambda)>0$ so that the following holds. Suppose $g(t)$ is a solution to the \KR flow on $B_{g_0}(p,1)\times [0,T]$ so that
\begin{enumerate}
\item $|\nabla Rm(h)|+|\Rm(h)|^{3/2}\leq 1$ on $B_{g_0}(p,1)$;
\item $|\mathrm{Rm}(g(t))|\leq at^{-1}$ on $B_{g_0}(p,1)\times (0,T]$;
\item $\Lambda^{-1} h\leq g(t)\leq \Lambda h$ on $B_{g_0}(p,1)$, $t\in [0,T]$.
\end{enumerate} 
Then  for all $x\in B_t(p,1/4),\; t\leq \hat T\wedge T$,
$$t|\Rm(x,t)|+t^{3/2}|\nabla\Rm|\leq C_1.$$
\end{lma}
\begin{proof}

We will consider the cutoff functions $\Phi(x, t)=\phi(\eta(x, t))$ and $\tilde \Phi= \phi(2\eta(x,t))$ as in the proof of Theorem \ref{local-Lip}, and we may choose $\hat T(n,a,\Lambda)$ sufficiently small so that $\Phi$ and $\tilde \Phi$ satisfy the same conditions there for all $t\leq T\wedge \hat T$.  The dependence of $a$ will in fact only appear in $\Phi$ and $\tilde \Phi$ and hence $\hat T$.

We follow the proof in \cite{ShermanWeinkove2012}. We will use $C_i$ to denote constants depending only on $n,\Lambda$ but not $a$.     

Denote $\Psi=\Gamma_{g(t)}-\Gamma_h$. Consider the function $F=t\Phi Q+L\cdot \tr_gh$ where $L$ is a constant to be chosen later and $Q=|\Psi|^2$. Using the evolution equation \eqref{genereal-equ}, we have 
\begin{equation}
\begin{split}
\heat F&=\Phi Q+t \left[\Phi \heat Q+Q \heat \Phi\right]\\
&\quad -2t{\bf Re}\left(g^{i\bar j} \Phi_i Q_{\bar j} \right)+L \heat \tr_gh \\
&\leq \Phi Q+ t\Phi \left[-|D \Psi|^2+C_2 Q \right]\\
&\quad +100t \Phi Q+C_ntQ^\frac{1}{2} |D \Psi||D \Phi|+ L\left( -\Lambda^{-1}Q +C_2\right)\\
&\leq Q\left(-L\Lambda^{-1}+C_3\right)+C_3\\
&\leq -Q+C_3.
\end{split}
\end{equation}
where the last inequality holds provided we choose $L=\Lambda(C_3+1)$.   By the maximum principle, we conclude that  $Q\leq C_3$ at a point where $F$ is maximal for $t\leq T\wedge \hat T$, and we conclude from this, the defintion of $F$ and the propoerties of the cutoff $\Phi$ that we have 
\begin{align}\label{intermediate-esti}
tQ(x,t)&\leq C_2.
\end{align}
on $B_t(p,\frac{1}{2})$  for all $t\leq T\wedge \hat T$

 Now we consider the function $$G=\frac{\wt\Phi t^2|\Rm|^2}{A-\wt F }$$ where $\tilde F:=tQ+L \tr_gh$. By \eqref{intermediate-esti} and our previous estimates, we may assume $A$ is sufficiently large so that $2A\geq A-\tilde F\geq \frac{1}{2}A$ on the support of $\wt\Phi$.   On the support of $\wt\Phi$ we have 
\begin{equation}
\begin{split}
\heat \wt F&\leq -t|\nabla \Psi|^2-t|\bar\nabla \Psi|^2+C_4.
\end{split}
\end{equation}
Thus we have
\begin{equation}\label{F0}
\begin{split}
&\quad \heat \left[ t^2|\Rm|^2(A-\tilde F)^{-1}\right]\\
&\leq (A-\tilde F)^{-1}\left[C_nt^2|\Rm|^3-t^2|\nabla \Rm|^2-t^2|\bar\nabla \Rm|^2 +2t|\Rm|^2\right]\\
&\quad +t^2|\Rm|^2(A-\tilde F)^{-2}\left[-t|\nabla \Psi|^2-t|\bar\nabla \Psi|^2+C_5 \right]\\
&\quad -2(A-\tilde F)^{-3} t^2|\Rm|^2|\nabla \tilde F|^2-2t^2(A-\tilde F)^{-2} {\bf Re}\left(g^{i\bar j}|\Rm|^2_i \cdot \tilde F_{\bar j} \right)\\
&\leq -C_6^{-1}t^3|\Rm|^4+C_6t^{-1}.
\end{split}
\end{equation}
Now suppose $(x_0,t_0)$ is the point where $G$ is maximal for $t\leq T\wedge \tilde T$.  Then either $t_0=0$ and then $G(t)$ thus $|Rm(g(t))|$ vanishes for all $t$, or else $t_0>0$ and from \eqref{F0} and the properties of $\wt \Phi$ and the fact that $\nabla G(x_0,t_0)=0$
\begin{equation}
\begin{split}
0\leq \heat G\leq C_7\tilde\Phi^{-1} G+C_7t^{-1}-C_7^{-1}\tilde\Phi t^3|\Rm|^4
\end{split}
\end{equation}
and hence $G(x_0,t_0)\leq C_8$.  In particular, we have shown that if $t\leq  T\wedge \hat T$, $x\in B_t(p,1/4)$, then 
\begin{align}
t|\Rm(g(t))|\leq C_9.
\end{align}
The first order estimate $|\nabla Rm(g(t))|$ follows from Shi's derivative estimate, see for example \cite{CaoChenZhu2008}.
\end{proof}

We also need the following local estimate from \cite[Propositon A.1]{LottZhang2013}.
\begin{lma}\label{l-curv1}
For any $A_1,n>0$, there exists $B(n,A_1)$ depending only on $n, A_1$ such that the following holds: For any \K manifold $(N^n,g_0)$ (not necessarily complete), suppose $g(t), t\in [0,T]$ is a solution of \K Ricci flow on $B_{0}(x_0,r) \subset N$ such that 
$$|Rm_{g(0)}|\leq A_1r^{-2} \quad\text{and }\quad |\nabla_{g_0} \Rm(g_0)|\leq A_1r^{-3}. $$
on $B_{0}(x_0,r)$ and
$$A_1^{-1}g_0\leq g(t)\leq A_1 g_0$$
  on $B_{0}(x_0,r)]\times[0, T]$.  Then we have
$$|\Rm|(g(t))\leq Br^{-2}$$
on $B_0(x_0,\frac r8)\times [0,T]$.
\end{lma}

%\section{Continuity of lower bound}
The next proposition shows that the global lower bound of a complete solution is continuous in the following sense 
\begin{prop}\label{prop1}
Let $g(t)$ be a smooth solution to \eqref{krf} on $M\times [0,T]$ such that $g(t)\in S(c_1, c_2,h)$ for all $t\in [0, T]$and some constants $c_i>0$.  If $g(0)\geq h$ then we have
\begin{equation}\label{dd1}g(t)\geq e^{-{nKt}/{c_1}} h\end{equation}
on $M\times [0,T]$.
\end{prop}
\begin{proof}
Suppose the curvatures of $h$ are bounded by $K$ in absolute values.  In particular, there exits an exhaustion function  $\tilde\rho$ of $M$ with $\tilde \rho\geq 1$ where $|\partial \tilde\rho|+|\ddb \tilde\rho|$ bounded on $M$.  We will let $\rho:=e^{At} \tilde \rho$ where the constant $A$ is large enough such that 
$$\heat \rho \geq 1$$
on $M\times[0, T]$.  Let $\e>0$ be given and consider the family of tensors $$\a(t)= (1+\e \rho )g(t)-(e^{-cKt}) h $$ on $M$ where $c=n/c_1$. Clearly, $\a(0)>0$ and $\a(t)\rightarrow +\infty$ as $x\rightarrow \infty$. We will show that $\a(t)>0$ on $M\times[0, T]$.  The proposition will then follow by letting $\e \to 0$.

Suppose on the contrary that $\a$ is not positive on $M\times [0,T]$.  Due to the presence of $\rho$, there is $(x_0,t_0)\in M\times (0,T]$ and $X\in T^{1,0}_{x_0}M$ with $|X|_{g(t_0)}=1$ such that 
$$\a_{X\bar X}=0.$$
We may choose $t_0>0$ such that for all $t<t_0$, $\a>0$ on $M$. Extend $X$ by parallel transport using metric $g(t_0)$ so that at $(x_0,t_0)$, $\partial_t X=\nabla X=\Delta X=0$. At $(x_0,t_0)$, we have 
\begin{equation}\label{d1}
\begin{split}
0&\geq \heat \a_{X\bar X}\\
&=-(1+\e \rho)R_{X\bar X}+e^{-cKt} \Delta h_{X\bar X}+cKe^{-cKt} h_{X\bar X}+\e\Box \rho
\end{split}
\end{equation}

On the other hand, using the fact that $\a_{X\bar Y}=0$ for all $Y\in T^{1,0}M$, we have
\begin{align*}
\Delta h_{X\bar X}
&=g^{p\bar q}h_{k\bar l} \Psi^k_{pX}\Psi^{\bar l}_{\bar q\bar X}- g^{p\bar q}\tilde R_{p\bar qX\bar X}+\frac{1}{2}\left(R_{\bar j}^{\bar l} h_{X\bar l}+R^k_i h_{k\bar X} \right)\\
&=g^{p\bar q}h_{k\bar l} \Psi^k_{pX}\Psi^{\bar l}_{\bar q\bar X}- g^{p\bar q}\tilde R_{p\bar qX\bar X}+e^{cKt}(1+\e\rho)R_{X\bar X}.
\end{align*}
where we denote $\tilde R$ to be the curvature of $h$. Hence, we have 
\begin{align*}
\heat \a_{X\bar X}&\geq cKe^{-cKt} h_{X\bar X} - e^{-cKt} g^{p\bar q} \tilde R_{p\bar qX\bar X}+\e \\
&\geq e^{-cKt}(cK- \frac{n}{c_1} K)+\e\\
&>0.
\end{align*}
But this contradicts \eqref{d1}.

\end{proof}

\section{A global existence result for Chern Ricci flow}

 We recall here a global existence result for the Chern Ricci flow from \cite{LeeTam2017}. 
 A family of hermitian metrics $g(t)$ on $M$ is said to be a solution to the Chern Ricci flow if 
 \begin{align}\label{crf}
 \frac{\partial}{\partial t} g_{i\bar j}=-R^C_{i\bar j}
 \end{align}
 where $R^C_{i\bar j}(g(t))$ is the Chern-Ricci curvature of $g(t)$.    If $g(t)$ is a family of \K metrics, then \eqref{crf} and \eqref{krf} coincide since in this case we have the Chern-Ricci curvature $R^C_{i\bar j}$ of $g(t)$ is the same as the Ricci curvature $R_{i\bar j}$ of $g(t)$.  In a local holomorphic coordinate, the Chern Ricci curvature is given by 
 $$R_{i\bar j}^C=-\partial_i \partial_{\bar j} \log \det g(t)=g^{k\bar l}R^C_{i\bar j k\bar l}$$
 where $R^C$ denotes the curvature tensor with respect to the Chern connection.  Similar to the \KR flow,  if $g(t)$ is a Hermitian family solving \eqref{crf} then we may write the corresponding family of forms $\omega(t)$ as
$$\omega(t)=\omega_0-t\Ric^C(h)+\ddb \varphi$$ 
 where $h$ is any Hermitian background metric on $M$ and the evolving potential $\varphi$ is given by 
 $$\varphi(t)=\int^t_0 \log\frac{\det g(s)}{\det h}ds$$
In particular, the Chern-Ricci flow will preserve K\"ahlerity on open set $U$ provided $g_0$ is initially \K on $U$.  From now on we will simply use $\Ric$ to denote the Chern-Ricci curvature of a Hermitian metric $g$, which will coincide with the Levi Civita Ricci curvature at any point $g$ is \K.  We have the following fundamental existence theorem for the Chern-Ricci flow.

\begin{thm}\cite[Theorem 4.2]{LeeTam2017}\label{t-existence-2}
Let $(M^n,g_0)$ be a complete noncompact Hermitian manifold and let $T$ be the torsion of $g_0$.
 Suppose the Chern curvature of $g_0$, $|T|^2_{g_0}$ and $|\bar\p T|_{g_0}$ are uniformly bounded by $K>0$. Suppose also the Riemannian curvature of $g_0$ is bounded.  Then there is a constant $\a(n)>0$ depending only on $n$ such that the Chern-Ricci flow has a smooth solution $g(t)$ on $M\times[0,\a K^{-1}]$ such that $\a g_0\le g(t)\le \a^{-1} g_0$ on $M\times[0,\a K^{-1}]$.
\end{thm}

\section{Proof of Main Theorem \ref{mainthm}}

We begin by proving a result on the existence of local solutions to \eqref{krf} satisfying certain estimates (Lemma \ref{mainclaim}).  For our purposes later on, it will be convenient to state the result in terms of the following constants.

\begin{defn}\label{defn1} 
Let $h$ be a complete \K metric satisfying  $|\nabla \Rm(h)|+|Rm(h)|^{3/2}\leq1$ on $M$ and let $g_0$ be a \K metric with $\lambda^{-1}h\leq g_0\leq \lambda h$.  Define the following numbers depending only on $n, \lambda$:
\begin{itemize}
\item Let $\Lambda(n, \lambda)=\max\{ \lambda,2 C_0\lambda^{c_1}\}$ where $C_0(n)$ and $c_1(n)$ be the constants obtained from Theorem \ref{local-Lip};
\item Let $C_1(n,\Lambda)$ be the constant from Lemma \ref{curv-esti} with $\Lambda$ from above;
\item Let $B(n,C_1+\a(n)^{-1})$ be the constant obtained from Lemma \ref{l-curv1} with $A_1=C_1+\a(n)^{-1}$;
\item Let $\mu(n,\lambda)=\sqrt{\left(1+\frac{\a(n)}{C_1+\b_n\Lambda^4b} \right)}-1$ where $\a(n)$ is from Lemma \ref{t-existence-2}, $C_1$ is from above, $b$ is from Lemma \ref{l-exhaustion-1} with $\kappa=0.1$ and $\b_n$ is a large dimensional constant to be specified in the proof;
\item Let $a(n,\lambda)=\max\left\{ 2B(1+\mu)^2,n\log\frac{\Lambda}{\a}\right\}$ ;
\item Let $\tilde T(n,a,\lambda)$ be the constant from Theorem \ref{local-Lip};
\item Let $\hat T(n,a,\Lambda)$ be the constant from Lemma \ref{curv-esti};
\item Let $\sigma(n,\lambda)=\max\{ 1, \tilde T^{-1/2},\hat T^{-1/2}\}$.
\end{itemize}
\end{defn}
By the bound on the curvature of $h$ and its gradient in Definition \ref{defn1}, by \cite{Tam2010} there is an exhaustion function $\rho$ on $M$ with 
\begin{align}
\label{exhaution-est}|\partial\rho|^2+|\ddb \rho|_h \leq 1.
\end{align}
 Let 
$$U_s=\{ x \in M: \;\rho(x) <s\} \subset M.$$
Our result on existence of local solutions can then be stated as
\begin{lma}\label{mainclaim}
 Under Definition \ref{defn1}, given any $R>>1$ there is a solution $g(t)$ to \KR flow \eqref{krf} defined on $U_{R-4\Lambda \mu (1+\mu)^{-1}}\times [0, \sigma^{-2} (1+\mu)^{-2}]$
with 
\begin{enumerate}
\item $|\Rm(g(t))|\leq at^{-1}$;
\item $|\varphi(t)|\leq at$.
\end{enumerate}
\end{lma}

\begin{proof}
We begin with the following

\begin{claim}\label{int-pf}  There is a short-time solution to the \KR flow on $U_R\times [0,t_0]$ for some $t_0>0$ sufficiently small and possibly depending on $R$, so that 
\begin{enumerate}
\item $|\Rm(x,t)|\leq at^{-1}$;
\item $|\varphi(x,t)|\leq at$.
\end{enumerate}
\end{claim}
\begin{proof}[proof of Claim \ref{int-pf}]
Let $F(x)=\mathfrak{F}(\frac{\rho(x)}{R+1})$ with $0<\kappa<1/8$ such that $(1+R)(1-\kappa)\geq R$ where $\mathfrak{F}, \kappa$ are from Lemma \ref{l-exhaustion-1}.
Consider the complete Hermitian metric $\tilde g_0=e^{2F}g_0$ on $U_{R+1}$.   It follows from Lemma \ref{l-exhaustion-1}  that $\tilde g_0$ will satisfy the hypothesis of Theorem \ref{t-existence-2} (see \cite{LeeTam2017}) and so 
we obtain a solution $ \tilde g(t)$ to Chern Ricci flow \eqref{crf}  on $U_{R+1}\times[0, t_0]$ for some $t_0>0$ with initial condition  $\tilde g_0$. Moreover, the restriction $g(t)=\tilde g(t)|_{U_R}$ solves \KR flow \eqref{krf} with initial condition $g_0$ because $\tilde g_0=g_0$ on $U_R$. Let $\tilde h=e^{2F}h$ and define
\begin{align}
\varphi=\int^t_0\log \frac{\det \tilde g(s)}{\det \tilde h}ds.
\end{align}
Then we may write $\tilde \omega(t)=\tilde \omega_0-t\Ric(\tilde h)+\ddb \varphi$. 

By smoothness of the solution, the curvature estimate in the claim follows by shrinking $t_0$ if necessary to be sufficiently small.  Moreover by \eqref{initial-equi}, by further shrinking $t_0$ if necessary we may have
$$\frac{1}{2\lambda^n}\leq  \frac{\det \tilde g(t)}{\det \tilde h}\leq 2\lambda^n$$
on $U_R\times [0,t_0]$ which implies the second conclusion in the claim.
\end{proof}

Now we provide the construction of an auxilliary function used in our constructions of local solutions to \KR flow. 

For $\kappa\in (0,1)$, let $f:(-\infty,1)\to[0,\infty)$ be the function:
\be\label{e-exh-1}
 f(s)=\left\{
  \begin{array}{ll}
    0, & \hbox{$s\in(-\infty,1-\kappa]$;} \\
    -\displaystyle{\log \lf[1-\lf(\frac{ s-1+\kappa}{\kappa}\ri)^2\ri]}, & \hbox{$s\in (1-\kappa,1)$.}
  \end{array}
\right.
\ee
Let   $\varphi\ge0$ be a smooth function on $\R$ such that $\varphi(s)=0$ if $s\le 1-\kappa+\kappa^2 $, $\varphi(s)=1$ for $s\ge 1-\kappa+2 \kappa^2 $
\be\label{e-exh-2}
 \varphi(s)=\left\{
  \begin{array}{ll}
    0, & \hbox{$s\in(-\infty,1-\kappa+\kappa^2]$;} \\
    1, & \hbox{$s\in (1-\kappa+2\kappa^2,1)$.}
  \end{array}
\right.
\ee
such that $\displaystyle{\frac2{ \kappa^2}}\ge\varphi'\ge0$. Define
 $$\mathfrak{F}(s):=\int_0^s\varphi(\tau)f'(\tau)d\tau.$$
Now we have the following Lemma from \cite{LeeTam2017}.
\begin{lma}\cite{LeeTam2017}\label{l-exhaustion-1} Suppose $0<\kappa<\frac18$. Then the function $\mathfrak{F}\ge0$ defined above is smooth and satisfies the following:
\begin{enumerate}
  \item [(i)] $\mathfrak{F}(s)=0$ for $0\le s\le 1-\kappa$.
  \item [(ii)] $\mathfrak{F}'\ge0$ and $\sum_{k=1}^2\exp( -k\mathfrak{F})|\mathfrak{F}^{(k)}|\leq b(\kappa)$.
  %\item [(iii)]  For any $ 1-2\kappa <s<1$, there is $\tau>0$ with $0<s -\tau<s +\tau<1$ such that
 %\bee
% 1\le \exp(\mathfrak{F}(s+\tau)-\mathfrak{F}(s-\tau))\le (1+c_2\kappa);\ \ \tau\exp(\mathfrak{F}(s_0-\tau))\ge c_3\kappa^2
% \eee
  %for some absolute constants  $c_2>0, c_3>0$.
\end{enumerate}

\end{lma}
Now define $t_k$ and $R_k$ inductively as follows
\begin{enumerate}
\item[(a)] $t_0$ is from claim \ref{int-pf} and $R_0=R$.
\item[(b)] $t_{k}=t_{k-1} (1+\mu)^2=t_0(1+\mu)^{2k}$;
\item[(c)]  $R_k=R_{k-1}-4\Lambda \sigma \sqrt{t_{k-1}}=R-4\Lambda \sigma (1+\mu)\mu^{-1}\sqrt{t_k} \left[ 1-(1+\mu)^{-k}\right]$.
\end{enumerate}
Consider the following statement 
 
$P(k):$ There is a solution of the \KR flow $g(t)$ on $U_{R_k}\times [0,t_k]$ with $g(0)=g_0$ such that 
\begin{enumerate}
\item $|\Rm(g(t))|\leq at^{-1}$;
\item $|\varphi(t)|\leq at$
\end{enumerate}
on $U_{R_k}\times [0,t_k]$ with $t_k\leq \sigma^{-2}$.

Clearly we see that $P(0)$ is true, while $t_k\rightarrow +\infty$ and $R_k \rightarrow -\infty$ as $k\rightarrow \infty$.  Let $k$ be the largest integer such that $P(k)$ is true. Then at $k+1$, we have the following possibilities.\\

{\bf Case 1: $t_{k+1} =(1+\mu)^2t_k\geq \sigma^{-2}> t_k$.} Then
\begin{equation}
\begin{split}
R_{k}&=R_{k-1}-4\Lambda \sigma\sqrt{t_{k-1}}\\
&=R-4\Lambda \sigma  \sqrt{t_{k}}\cdot \sum_{i=1}^{k} \frac{1}{(1+\mu)^{i}}\\
&\geq R-\frac{4\Lambda \mu}{1+\mu}.
\end{split}
\end{equation}
and the Lemma holds in this case.

{\bf Case 2: $t_{k+1}<\sigma^{-2}$.} We will show that this is not possible. Let $g(t)$ be the solution to \eqref{krf} on $U_{R_k}\times [0,t_k]$ from the statement $P(k)$. Using $\lambda^{-1}h\leq g_0\leq \lambda h$, for $x\in U_{R_k-\Lambda \sigma \sqrt{t_k}}$ we have 
$$B_{g_0}(x,\sigma \sqrt{t_k})\subset U_{R_k}.$$
Let $r=\sigma \sqrt{t_k}<1$. By the choice of $\sigma$ from definition \ref{defn1}, we can apply Theorem \ref{local-Lip} to \KR flow $\tilde g(t)=r^{-2}g(r^2t)$ on $B_{\tilde g(0)}(x,1)\times [0,t_k r^{-2}]$ with reference metric $\tilde h=r^{-2}h$ to deduce that for all $(x,t)\in U_{R-\Lambda\sigma \sqrt{t_k}}\times [0,t_k]$, 
\begin{align}\label{new-equ}
\Lambda^{-1} h \leq g(t)\leq \Lambda h.
\end{align}
Here we use the fact that $r<1$ so that the rescaled reference metric $\tilde h$ satisfies the curvature assumptions made on $h$ in definition \ref{defn1}.  Thus for $x\in U_{R_k-2\Lambda \sigma\sqrt{t_k}}$, we apply Lemma \ref{curv-esti} to $\tilde g(t)=r^{-2}g(r^2t)$ on $B_{\tilde g(0)}(x,1)\times [0,t_kr^{-2}]$ with reference metric $\tilde h$ to show that for $(x,t)\in  U_{R-2\Lambda\sigma \sqrt{t_k}}\times [0,t_k]$, 
\begin{align}\label{imp-cur}
t|Rm(x,t)|+t^{3/2}|\nabla Rm|\leq C_1.
\end{align}

At $t=t_k$, we extend the \KR flow as follows. Fix $\kappa\in (0,\frac{1}{10})$, let 
$$F(x)=\mathfrak{F}\left( 1-\kappa+\frac{\kappa}{\Lambda \sigma \sqrt{t_k}}\left(\rho(x)-R_k+3\Lambda \sigma\sqrt{t_k}\right)\right)$$
and consider the conformal change of metric $\hat g_0=e^{2F}g(t_k)$ and $\hat h=e^{2F}h$. By \cite[Lemma 4.3]{LeeTam2017}, $\hat g_0$ and $\hat h$ are complete Hermitian metric with bounded geometry of $\infty$ order on $U_{R_k-2\Lambda \sigma\sqrt{t_k}}$ with $\hat g_0=g(t_k)$ on $U_{R_k-3\Lambda\sigma\sqrt{t_k}}$. Moreover, by \cite[Lemma B1]{LeeTam2017}, \eqref{exhaution-est}, \eqref{new-equ} and the fact that $r<1<\Lambda $, there is $C'(n)>0$ such that 
\begin{align}
|\Rm^C(\hat g_0)|+|\partial \hat\omega_0|^2+|\nabla_{\bar\partial}\partial \hat\omega_0|\leq \frac{C_1+C'(n)\Lambda^4 b}{t_k}=K.
\end{align}
We choose $\b(n)$ in definition \ref{defn1} to be $C'(n)$ here. By Theorem \ref{t-existence-2}, there is a solution to the \KR flow $\hat g(t)$ starting from $\hat g_0$ on $U_{R_k-2\Lambda \sigma\sqrt{t_k}}\times [0,\a(n)K^{-1}]$ with 
\begin{align}\label{new-equ-2}
\a\hat g_0\leq \hat g(t)\leq \a^{-1}\hat g_0.
\end{align}
Restrict $\hat g(t)$ to $U_{R_k-3\Lambda \sigma\sqrt{t_k}}$ provides solution $g(t)=\hat g(t-t_k)$ to \KR flow on $U_{R_k-3\Lambda \sigma\sqrt{t_k}}\times [0,t_{k+1}]$.

Moreover using \eqref{new-equ-2}, \eqref{imp-cur} we apply Lemma \ref{l-curv1} on $B_{g(t_k)}(x, \sqrt{t_k})$ where $x\in U_{R_k-4\Lambda \sigma\sqrt{t_k}}$ to deduce that for all $t\in [0,t_{k+1}]$, 
\begin{align}
|\Rm(x,t)|\leq \frac{B}{t_k}< \frac{a}{t}.
\end{align}
The last inequality is by our choice of $a$ in definition \ref{defn1} and properties from $P(k)$. On the other hand, using \eqref{new-equ} and \eqref{new-equ-2} for $t\in [t_k,t_{k+1}]$, 
\begin{equation}
\begin{split}
|\varphi(t)|\leq at_k + n(t-t_k) \log \frac{\Lambda}{\a}<at.
\end{split}
\end{equation}
We have thus shown that in this case $P(k+1)$ is true, which contradicts the maximality of $k$.

 Thus Case 2 above cannot occur, leaving only Case 1 in which case the Lemma holds.  Thuis completes the proof of the Lemma.

\end{proof}

We now prove Theorem \ref{mainthm} using Lemma \ref{mainclaim}.

\begin{proof}[Proof of Theorem \ref{mainthm}]
 
By the results in \cite{Shi1997} and by scaling, there is a \K metric $\tilde h$ having curvature and gradient of curvature  bounded on $M$ by $1$.  Also we may have $C^{-1}(n, c, K) \tilde h  \leq g_0 \leq C(n, c, K)  \tilde h$ on $M$ for some constant $C(n, c, K)$ where $c, K$ are from the hypothesis of Theorem \ref{mainthm}.  It follows by Lemma \ref{mainclaim} that there is a solution $g_k(t)$ to \KR flow \eqref{krf} defined on $B_{g_0}(k) \times [0, S(n,c, K)]$ for all $k$ sufficiently large satisfying the conditions in the Lemma.  Theorem \ref{local-Lip} and local Evans-Krylov theory \cite{Evan1982,Krylov1983} or \KR flow local estimates \cite{ShermanWeinkove2012}  then imply the existence of $\tilde T(n, c, K)$ such that some subsequence of $g_k(t)$ converges smoothly and locally uniformly to a solution $g(t)$ to \eqref{krf} on $M\times[0, S \wedge \tilde T]$ such that $g(t) \in S(c_1, c_2)$ for all $t$ and some $c_i$ and satisfying the conditions in the claim.  The existence of the constants $a, T, C_2, C_1$ in Theorem \ref{mainthm} follows from this, and  Proposition \ref{prop1}.  The fact that $g(t)$ extends as a bounded curvature solution to \eqref{krf} on $M\times[0, T_h)$ follows from condtion (1) in the Theorem and \cite[Theorem 2.2]{CLT2} (see also \cite[Theorem 1.3]{LeeTam2017}).

\end{proof}

\section{Proof of Theorem \ref{non-smooth} }  Theorem \ref{non-smooth} (1) will follow from the following slightly more general existence theorem.
\begin{thm}\label{non-smoothnew}
Let $g_0$ be a continuous Hermitian metric. Suppose there is a sequence of smooth \K metric such that 
\begin{enumerate}
\item  $h_{k,0} \in S(1,c_k,h)$ for some $c_k$;
\item for any $\Omega\subset \subset M$, $h_{k,0}\rightarrow g_0$ uniformly on $\Omega$;
\item for any $\Omega\subset \subset M$, the scalar curvature of $h_{k,0}$, $R_k\geq -C(\Omega)$ on $\Omega$. 
\end{enumerate}
Then there is $T(n,h)>0$ so that \eqref{krf} has a solution $g(t)$ on $M\times(0,T]$ with $g(t)\geq \frac{1}{2n}h$.
\end{thm}

 We begin by recalling the following Lemma from \cite{CLT1} which basically says that if a local solution $h(t)$ to
\eqref{krf}  is a priori uniformly equivalent to a fixed metric $\hat{g}$ in
space time, and close to $\hat{g}$ at time $t=0$, then it remains close to
$\hat{g}$
in a uniform space time region.  

\begin{lma}\label{lemgequivalenttoboundedcurvaturelocal}
Let $h(t)$ be a smooth solution to \eqref{krf} on $B(1)\times [0,
T)$ with $h(0)=h_0$ where $B(1)$ is the unit Euclidean ball in $\C^n$. Let
$\hat g$ be a smooth \K metric on $B(1)$.   Suppose
\be\label{lemgequivalenttoboundedcurvaturelocale1}
  N^{-1}\hat{g}\leq h(t)\le N\hat g
\ee
on $B(1)\times [0,T)$ for some $N>0$, and that
\be\label{lemgequivalenttoboundedcurvaturelocale2}
 (1-d)\hat{g}\leq h_0\le (1+d)\hat g
\ee
on $B(1)$.  Then there exists a positive continuous function $a(t):[0, T)\to \R$
depending only on $\hat{g}, N, d$ and $n$ such that
\be\label{lemgequivalenttoboundedcurvaturelocale3}
 \frac{(1-d)(1-a(t))}{(1+d)}h_0\leq h\le (1+a(t))h_0
\ee
on $B(1/2) \times[0,T)$, where $\lim_{t\to 0} a(t)=n\sqrt{2d(1+d)}/(1-d)$.
\end{lma}

\begin{proof}[Proof of Theorem \ref{non-smoothnew}]
Let $g_0$ be as in Theorem \ref{non-smooth}.  Thus there is a sequence $\{h_k\}\subset S(1, c_k, h)$ for some sequence  $c_k$ such that $h_k \to g_0$ in $C^{0}_{loc}(M)$.   By Theorem \ref{mainthm}, there is a sequence of solutions $h_k(t)$ to \eqref{krf} on $M\times [0,T_{h})$ and stays uniformly equivalent to $h$ for $t\leq T(k)$. By applying maximum principle on equation \eqref{genereal-equ} of $\tr_{h_k(t)}h$, there is $\e_n>0$ such that for all $(x,t)\in M\times  [0,\e_n K^{-1}]$, 
$$h_k(t)\geq \frac{1}{2n}h.$$
By \cite[Lemma 3.3]{CLT1}, we also have local uniform upper bound for $h_k(t)$ with respect to $h$ on $[0,\e_nK^{-1}]$. In particular, by the Evans-Krylov theory \cite{Evan1982,Krylov1983} or \KR flow local estimates \cite{ShermanWeinkove2012},  we may conclude subequence convergence $h_{k_i}(t)\to g(t)$ in $C^{\infty}_{loc}(M\times(0,\e_n K^{-1}]$ where $g(t)$ solves \eqref{krf}. By applying Proposition \ref{prop1} on each $h_k(t)$, it is easy to see that
$$g(t)\geq e^{-C_1Kt}h.$$

 It remains to prove $g(t)\to g_0$ in $C^{0}_{loc}(M)$. Fix some holomorphic coordinate ball $B(1)$ on $M$ and let $\delta>0$ be given.  Then we may choose some constant $k_0$ sufficiently large and some constant $C'(n, c_k, K)>0$ so that on $B(1)$ we have the following 

\begin{equation}\label{eee1}
 \left\{
 \begin{array}{ll}
(1-\delta)g_0&\leq h_{k_0} \leq (1+\delta)g_0\\
(1-\delta)h_{k_0}&\leq h_k \leq (1+\delta)h_{k_0} \text{\,\,\, for all $k\geq k_0$}\\
C'^{-1} h_{k_0}&\leq h_k(t) \leq C' h_{k_0} \text{\,\,\, for all $k\geq k_0$ and $t\in [0, T)$}\\
  \end{array}
 \right.
\end{equation}  
for some constant $C'$ depends only on $n, c, K$.  Then by Lemma \ref{lemgequivalenttoboundedcurvaturelocal} and \eqref{eee1} there is a function $a(t):[0, T)\to \R$ depending only on $h_{k_0}, C'(n, c, K), \delta, n$ such that  

\begin{equation}\label{eeeee1}
\lim_{t\to 0} a(t)=n\sqrt{2\delta(1+\delta)}/(1-\delta)
\end{equation}
and on $B(1)\times[0, T)$ we have
\begin{equation}\label{eeee1}
\frac{(1-\delta)(1-a(t))}{(1+\delta)} (1-\delta)^2 g_0 \leq h_k(t)  \leq (1+a(t))(1+\delta)^2 g_0
\end{equation} 
for all $k\geq k_0$.

Using \eqref{eeeee1} and \eqref{eeee1}, we conclude that given any $\epsilon>0$ we may choose $k_0, t_0>0$ so that for all $k\geq k_0$ and $0<t<t_0$ we have $|h_k (t) -g_0|_{g_0} \leq \epsilon$ on $B(1)$.  It follows that $g(t)\to g_0$ in $C^{0}_{loc}(M)$ which completes the proof of Theorem \ref{non-smoothnew}.
\end{proof}

 Theorem \ref{non-smooth} (1) follows directly from Theorem \ref{non-smoothnew}.  We now prove  (2).  Let $g_0$ be as in Theorem \ref{non-smooth} (2).  Thus there is a sequence $\{h_k\}\subset S(1, c, h)$ such that $h_k \to g_0$ in $C^{0}_{loc}(M)$, and by Theorem \ref{mainthm}, there is a sequence of solutions $h_k(t)$ to \eqref{krf} on $M\times [0,T_{h})$ each satisyfing the conclusions in  Theorem \ref{mainthm}.  In particular, by the uniform estimates in Theorem \ref{mainthm}, there exists $T(n, c, K)$ such that given any $T<T(n, c, K)$ we may have $C^{-1}h\leq h_k(t)\leq Ch$ on $M\times [0,T)$ for some some $C$ independent of $k$.  Thus by the Evans-Krylov theory \cite{Evan1982,Krylov1983} or \KR flow local estimates \cite{ShermanWeinkove2012}, we may conclude subequence convergence of $h_{k_i}(t)\to g(t)$ in $C^{\infty}_{loc}(M\times(0,T(n, c, K))$ where $g(t)$ solves \eqref{krf}.  The fact that we may extend this solution to $M\times(0, T_h)$ follows from the same argument as in Theorem \ref{mainthm}.  The fact that $g(t)\to g_0$ as $t\to 0$ follows by  Lemma \ref{lemgequivalenttoboundedcurvaturelocal}  exactly as in the proof of part (1).

\section{Proof of Theorem \ref{Stability}}
\begin{proof}[Proof of Theorem \ref{Stability}]
Let $g_0\in S(c_1,c_2,h)$, by Theorem \ref{mainthm} and By \cite[Theorem 1.3]{LeeTam2017} there is a longtime solution $g(t)$ to \eqref{krf} on $M\times [0,+\infty)$ with $g(0)=g_0$ such that for any $[a,b]\subset [0,+\infty)$, $g(t)$ is uniformly equivalent to $h$. If $k<0$, the longtime convergence follows from \cite[Theorem 1.1]{HuangLeeTamTong2018} and the uniqueness of complete \KE metric. It suffices to discuss the case of $k=0$.

For simplicity, we work on the universal cover $\mathbb{C}^n$. According to the parabolic Schwarz Lemma \eqref{genereal-equ} with $h$ being the flat metric, we can apply maximum principle to deduce that for all $t>0$, $g(t)\geq nc_1 h$. Also, since $\heat \log {\det g}=0$ and $\sup_{M\times[0, T)}\log \det g(x,t)<\infty$ for any $0<T<\infty$, we can apply the maximum principle to show $\log \det g$ is uniformly bounded above on $\mathbb{C}^n\times [0,+\infty)$ and hence further infer that $\exists C>1$ such that for all $(x,t)\in \mathbb{C}^n\times [0,+\infty)$,
\begin{align}\label{long-equi}
C^{-1}h\leq g(t)\leq Ch.
\end{align}
By Evans-Krylov estimates \cite{Evan1982,Krylov1983} or \KR flow local estimates \cite{ShermanWeinkove2012}, $g(t_k)$ converges to some \K metric $g_\infty$ on $\mathbb{C}^n$ for some subsequence $t_k\rightarrow +\infty$ in $C^\infty_{loc}$.

On the other hand, for $L$ sufficeintly large the function $F=t|\Psi|_{g(t)}^2+L\tr_gh $ satisfies
\begin{equation}
\begin{split}
\heat F&\leq 0.
\end{split}
\end{equation}
on $\mathbb{C}^n\times [0,+\infty)$.  Since $\sup_{M\times[0, T)}F(x, t)<\infty$ for any $0<T<\infty$, we can apply the maximum principle to show that $F$ is uniformly bounded above and hence $|\partial g(t)|\leq C't^{-1/2}$ on $\mathbb{C}^n\times [0,+\infty)$ for some $C'>0$. In particular, $\partial g_\infty=0$ implying $\omega_\infty=\ddb f$ for some quadratic polynomial $f$ and hence $\omega_\infty$ is a flat metric.
\end{proof}

\section{Proof of Theroem \ref{instantaneous}}
\begin{proof}
The proof here is similar to that in \cite{HuangLeeTam2019} except that here we have a \K approximation of the initial metric. For the sake of completeness, we present the proof here. For $\e>0$, let $g_{\e,0}=g_0+\e h$ be a complete \K metric uniformly equivalent to $h$. By Theorem \ref{mainthm}, there is a solution to the \KR flow $g_\e(t)$ on $M\times[0, T_\e)$ starting from $g_{\e,0}$ such that $g_\e(t)$ is uniformly equivalent to $h$ for all $t$, and we assume $T_\e$ is the maximal such time.  Consider the corresponding potential flow $\varphi_\e(t)=\int^t_0 \log \frac{\det g_\e(r)}{\det h}\;dr$ where we can rewrite the \K form of $g_\e(t)$ to be $\omega_\e(t)=\omega_0+\e\omega_h-t\Ric(h)+\ddb \varphi_\e$. 

\begin{claim}\label{approx-com}
Then there exists  $C(n,s,K,\b,f)>0$ such that for all $\e>0$ and $t\in (0,s\wedge T_\e)$ we have
\begin{enumerate}
\item[(i)] $\varphi_\e\leq Ct$,
\item[(ii)]$\varphi_\e\geq nt\log t-Ct$,
\item[(iii)] $\dot\varphi_\e \leq C$,
\end{enumerate}

\end{claim}
\begin{proof}[proof of claim \ref{approx-com}]
Since $h$ has bounded curvature and $g_\e(t)$ is uniformly equivalent to $h$, there is an exhaustion function $\rho>0$ on $M$ such that $|\ddb \rho|_h \leq C$. For any $\delta>0$, consider $\tilde\varphi=\varphi_\e-\delta \rho$ on $M\times [0,s\wedge T_\e-\delta]$.  Then at a point  $(x_0, t_0)\in  M\times [0,s\wedge T_\e-\delta]$ where the function $\tilde\varphi-\tilde Ct$ attains a maximum, we may calculate the following provided $t_0>0$ and $\tilde C$ is chosen sufficently large

\begin{equation}
\begin{split}
0&\leq \varphi_\e'-\tilde C\\
&\leq \log\left( \frac{\omega_0+\e\omega_h-t\Ric(h)+\delta\ddb \rho}{\omega_h^n}\right)-\tilde C\\
&<0.
\end{split}
\end{equation}
which is impossible. Hence $t_0=0$. 
Here we have used the $|\Rm(h)|\leq K$ and $\e,\delta<<1$. By letting $\delta\rightarrow 0$, we see that (i) holds.

For the lower bound, we consider $\hat \varphi=\varphi_\e+\delta \rho-tf-nt(\log t-1)+\tilde Ct$ on $[0,s\wedge T_\e-\delta]$ for some constant $\tilde C$. Then at a point  $(x_0, t_0)\in  M\times [0,s\wedge T_\e-\delta]$ where this function attains a minimum we calculate the following provided $t_0>0$ $\tilde C$ is sufficiently large, and $\delta$ is sufficiently small compared to $\e$:
\begin{equation}
\begin{split}
0&\geq \hat\varphi'\\
&=\varphi_\e' -f -n\log t+\tilde C\\
&\geq \log \frac{(\omega_0+\e\omega_h-t\Ric(h)+t\ddb f-\delta \ddb \rho )^n}{\omega_h^n}-f -n\log t+\tilde C\\
&\geq -n\log s+n\log \b-\sup_M |f|+\tilde C\\
&>0.
\end{split}
\end{equation}
which is impossible. Hence $t_0=0$. By letting $\delta\rightarrow 0$, we see that (ii) holds. 

 To derive the upper bound on $\dot\varphi_\e$ we use the fact $$(\frac{\partial}{\partial t}-\Delta_t ) \dot\varphi_\e= -\tr_{\omega(t)}\Ric(h)$$ 
and thus
 $$(\frac{\partial}{\partial t} -\Delta) (t\dot \varphi_\e -\varphi_\e -nt)= - \tr_{\omega} (\omega_{\e, 0})$$
and apply the maximum principle as above to show that the function $t\dot\varphi_\e-\varphi_\e-nt-\delta \rho$
attains its maximum at on $M\times [0,s\wedge T_\e-\delta]$ when $t=0$ provided $\delta$ is sufficiently small and the bound in (iii) follows by letting $\delta\to 0$ this and the bound in (i).

\end{proof}

\begin{claim}\label{approx-com2}
Then exists  $C(n,s,K,\b,f, t)>0$ such that for all $\e>0$ and $t\in (0,s\wedge T_\e)$ we have $$tr_{g_{\e}(t)} h\leq C$$
\end{claim}
\begin{proof}
  From the above we may calculate that the function $v=(s-t)\dot\varphi_\e+\varphi_\e-sf+nt$ satisfies
\begin{align}
\heat v=\tr_{g_\e} (\omega_0+\e\omega_h-s\Ric(h)+ s\ddb f) \geq \b\tr_{g_\e} h.
\end{align}
Then for sufficiently large $L>>1$ and using \eqref{genereal-equ} the function $F=\log \tr_{g_\e}h -Lv +(ns+1)\log t$ will satisfy
\begin{equation}
\begin{split}
\heat F&< \frac{ns+1}{t}-\tr_{g_\e}h.
\end{split}
\end{equation}
Now let $(x_0, t_0)$ be the point in $M\times [0,s\wedge T_\e-\delta]$ where the function  $F_\delta=F-\delta e^{At}\rho$ is maximal where $A$ is chosen large enough so that $\heat (e^{At}\rho)\geq 0$.  Now if $t_0=0$ then the claim follows from the estimates in Claim \ref{approx-com} and the hypothesis on $g_0$.  If $t_0>0$ then by the maximum principle we have  $0\leq \heat F_\delta (x_0, t_0) \leq \heat F(x_0, t_0)$ and thus  $t_0\tr_{g_\e}h(x_0, t_0) \leq m$ from the estimates in Claim \ref{approx-com}. We may assume $\tr_{g_\e}h(x_0, t_0) \geq 1$  as otherwise the claim trivially holds. Then  we have the following where the RHS is evaluated at t $(x_0, t_0)$
\begin{equation}
\begin{split}
F_\delta(x, t)&\leq  (ns+1)\log t_0+\log \tr_{g_\e}h +(s-t_0)\log \frac{\det g_\e}{\det h}+C_1(n,s,f,K)\\
&\leq (ns+1)\log t_0+\left(ns+1 \right)\log \tr_{g_\e}h+C_2\\
&\leq C_3.
\end{split}
\end{equation}
and the claim follows by letting  $\delta\rightarrow 0$.
\end{proof}

By the Claims above we conclude that for any $[a,b]\subset (0,s\wedge T_\e)$, there is $C_4(n,s,f,K,a,b)>1$ such that  
$$C_4^{-1}h\leq g_\e(t)\leq C_4h$$
and it follows from local estimate \cite{Evan1982,Krylov1983,ShermanWeinkove2012} and diagonal subsequence argument that we in fact have $T_{\e} \geq s$ for all $\e$ and further that we can let $\e\rightarrow 0$ to obtain a smooth solution $\varphi$ to \eqref{potential-flow-equ} on $M\times (0,s)$.  By Claim \ref{approx-com}, $\varphi(t)\rightarrow 0$ uniformly on $M$. The smoothness on $U=\{x: g_0>0\}$ follows from \cite[Theorem 1.2]{HuangLeeTam2019} since the argument there is purely local once we have the estimates from claim \ref{approx-com}. One can also argue using the pseudolocality of Ricci flow \cite{Perelman2002,HeLee2018}, see \cite[Remark 3.1]{HuangLeeTam2019}.

\end{proof}

\end{document}